\documentclass[12pt]{article}      
\usepackage{amsmath,amssymb,amsthm, amscd, graphicx, verbatim, setspace,mathrsfs} 
\usepackage{setspace}
\usepackage{mathtools}
\usepackage{dsfont}
\usepackage{tabstackengine}
\newcommand{\Mod}[1]{\ (\mathrm{mod}\ #1)}
\theoremstyle{plain}
\newtheorem{thm}{Theorem}
\newtheorem{lem}[thm]{Lemma}
\newtheorem{prob}[thm]{Problem}
\newtheorem{que}[thm]{Question}
\newtheorem{cor}[thm]{Corollary}

\usepackage{algorithm}
\usepackage{algpseudocode}
\usepackage{amsthm}
\usepackage{amsmath}

\title{Further results on discrete unitary invariance}
\begin{document}
\author{Jesse Geneson\\
Department of Mathematics, PSU, State College, PA 16801 USA \\ geneson@gmail.com}
\date{}
\maketitle

\begin{abstract}
In arXiv:1607.06679, Marcus proved that certain functions of multiple matrices, when summed over the symmetries of the cube, decompose into functions of the original matrices. In this note, we generalize the results from the Marcus paper to a larger class of functions of multiple matrices. We also answer a problem posed in the Marcus paper.
\end{abstract}

Marcus \cite{marc} exhibited certain functions that take multiple matrix inputs and are invariant with respect to unitary conjugation among subsets of input matrices. These formulas generalize results in \cite{marc1,marc2} that were used to develop a finite definition of free probability, and similar formulas were used in \cite{marc3} to prove the existence of Ramanujan graphs of all sizes and degrees.

The formulas in \cite{marc} are for sums of determinants and sums of products of pairs of determinants. In this paper, we extend the formulas to sums of products of $d$-tuples of determinants for all $d$. We also solve the following problem posed at the end of \cite{marc}:

\begin{prob}
Find a formula for $\sum_{P \in \mathscr{P}_{n}} [P]_{S,T}[P^{-1}]_{U,V}$ that holds for all $S, T, U, V \in \binom{[n]}{k}$.
\end{prob}

In Section \ref{formula}, we solve the problem above. In Section \ref{gene}, we generalize the results from \cite{marc}.

\section{A formula for $\sum_{P \in \mathscr{P}_{n}} [P]_{S,T}[P^{-1}]_{U,V}$}\label{formula}

As in \cite{marc}, define $\mathscr{P}_{n}$ to be the class of $n \times n$ permutation matrices over a commutative ring $\mathbb{K}$.

Given $X \subset [n]$ and $t \in X$, let $r(t, X)$ be the number of elements of $X$ that are less than $t$. Let $s(X,Y) = \sum_{t \in X \wedge t \notin Y} r(t, X)-\sum_{t \in Y \wedge t \notin X} r(t, Y)$.

In \cite{marc}, Marcus asked about a formula for $\sum_{P \in \mathscr{P}_{n}} [P]_{S,T}[P^{-1}]_{U,V}$ for all $S, T, U, V \in \binom{[n]}{k}$. We find a formula below.

\begin{lem}\label{main}
If $S, T, U, V \in \binom{[n]}{k}$ and $|S \cap V| = j$, then $\sum_{P \in \mathscr{P}_{n}} [P]_{S,T}[P^{-1}]_{U,V} = j! (n-2k+j)! (-\mathds{1})^{s(T,U)+s(S,V)}\delta_{\left\{|S \cap V| = |T \cap U| \geq k-1\right\}}$.
\end{lem}

\begin{proof}
We write $\sum_{P \in \mathscr{P}_{n}} [P]_{S,T}[P^{-1}]_{U,V} = \sum_{\pi \in S_{n}} [P_{\pi}]_{S,T}[P_{\pi}]_{V,U}$. For $\pi \in S_{n}$, let $\pi(U)$ denote the image of the elements of $U$ under $\pi$. As observed in \cite{marc}, $[P_{\pi}]_{U,V} = \emptyset$ whenever $\pi(U) \neq V$. Thus the terms of the sum will be zero unless the following conditions (*) are satisfied: $|S \cap V| = |T \cap U|$, $\pi(S \cap V) = T \cap U$, $\pi(S-(S \cap V)) = T-(T \cap U)$, and $\pi(V-(S \cap V)) = U-(T \cap U)$. 

If $|S-(S \cap V)| \geq 2$, then the terms of the sum will add to zero since there will be the same number of $1$ and $-1$ terms in the sum. This is because $[P_{\pi}]_{S,T} = -[P_{\pi'}]_{S,T}$ for every $\pi$ that satisfies the conditions (*), where $\pi'$ is the permutation obtained from $\pi$ by transposing the values of $\pi$ for the least two elements of $S-(S \cap V)$, and $\pi'$ also satisfies the conditions (*).

Therefore the sum will be zero unless $|S \cap V| = |T \cap U| \geq k-1$, $\pi(S \cap V) = T \cap U$, $\pi(S-(S \cap V)) = T-(T \cap U)$, and $\pi(V-(S \cap V)) = U-(T \cap U)$. There are $j! ((k-j)!)^2 (n-2k+j)!$ ways to select $\pi$ such that $|S \cap V| = |T \cap U| =  j$, $\pi(S \cap V) = T \cap U$, $\pi(S-(S \cap V)) = T-(T \cap U)$, and $\pi(V-(S \cap V)) = U-(T \cap U)$. Since the sum is zero unless $j \geq k-1$, we can simplify $j! ((k-j)!)^2 (n-2k+j)!$ to $j! (n-2k+j)!$. 

If $j = k$, then the formula in the lemma reduces to the formula proved in \cite{marc} since $s(T,U)+s(S,V) = 0$ when $j = k$. Otherwise if $j = k-1$, then each nonzero term of the sum is equal to $(-\mathds{1})^{s(T,U)+s(S,V)}$ since $(P_{\pi})_{S,T}$ can be transformed into $(P_{\pi})_{V,U}$ using $s(S,V)$ row swaps and $s(T,U)$ column swaps.
\end{proof}

\section{Generalizing the results from arXiv:1607.06679}\label{gene}

As in \cite{marc}, define $\mathscr{Q}_{n}$ to be the class of $n \times n$ sign matrices over the commutative ring $\mathbb{K}$. Note that a sign matrix is a diagonal matrix with all diagonal entries equal to either $\mathds{1}$ or $-\mathds{1}$.

We use the following lemma from \cite{marc} to prove a generalization of the determinant sums in \cite{marc}.

\begin{lem} \label{mrc}
\cite{marc} For $A \in \mathbb{K}^{m \times n}$ and $B \in \mathbb{K}^{n \times p}$, we have $[A B]_{S,T} = \sum_{U \in \binom{[n]}{k}} [A]_{S,U}[B]_{U,T}$ for all $S \in \binom{[m]}{k}$ and $T \in \binom{[p]}{k}$.
\end{lem}

The next lemma generalizes Section $4$ of \cite{marc}, which had formulas for the cases when $d \leq 2$.

\begin{lem}\label{gen}
For $j = 0, \ldots, d-1$, let $A_{j} \in \mathbb{K}^{n_{j} \times n_{j+1 \Mod{d}}}$, $B_j \in \mathbb{K}^{p_j \times n_j}$, and $C_j \in \mathbb{K}^{n_{j+1 \Mod{d}} \times r_j}$.
Then $\sum_{Q_j \in \mathscr{Q}_{n_j}: j = 0, \dots, d-1} \sum_{P_j \in \mathscr{P}_{n_j}: j = 0, \dots, d-1}$\\
$\Pi_{j = 0}^{d-1} [B_j (Q_j P_j) A_j (Q_{j+1 \Mod{d}}P_{j+1 \Mod{d}})^{-1}C_j]_{X_j,Y_j} =$\\
$\delta_{\left\{k_0 = k_1 = \dots = k_{d-1}\right\}} (\mathds{1}+\mathds{1})^{\sum_{j = 0}^{d-1} n_j} (\Pi_{j = 0}^{d-1} k_j! (n_j-k_j)! [B_j C_{j-1 \Mod{d}}]_{X_j,Y_{j-1 \Mod{d}}}) [\Pi_{j = 0}^{d-1} A_j]^{(k_0)}$ for all $X_j \in \binom{[p_j]}{k_j}$, $Y_j \in \binom{[r_j]}{k_j}$.
\end{lem}

\begin{proof}
By Lemma \ref{mrc}, $[B_j (Q_j P_j) A_j (Q_{j+1 \Mod{d}}P_{j+1 \Mod{d}})^{-1}C_j]_{X_j,Y_j} =$\\
$\sum_{S_{1,j},T_{1,j} \in  \binom{[n_j]}{k_j},S_{2,j},T_{2,j} \in \binom{[n_{j+1 \Mod{d}}]}{k_{j+1 \Mod{d}}}}$\\ 
$[B_j]_{X_j,S_{1,j}}[Q_j]_{S_{1,j},T_{1,j}}[P_j A_j P_{j+1 \Mod{d}}^{-1}]_{T_{1,j},S_{2,j}}[Q_{j+1 \Mod{d}}]_{S_{2,j},T_{2,j}}[C_j]_{T_{2,j},Y_j}$. 

Note that $\sum_{Q_j \in \mathscr{Q}_{n_j}: j = 0, \dots, d-1} \Pi_{j = 0}^{d-1} [Q_j]_{S_{1,j},T_{1,j}} [Q_{j+1 \Mod{d}}]_{S_{2,j},T_{2,j}} =$\\
$(\mathds{1}+\mathds{1})^{\sum_{j = 0}^{d-1} n_j} \Pi_{j = 0}^{d-1} \delta_{\left\{S_{1,j} = T_{1,j} = S_{2,j-1 \Mod{d}} = T_{2,j-1 \Mod{d}}\right\}}$.

In order to have all the $\delta$s satisfied, we must have $k_0 = k_1 = \dots = k_{d-1}$. \\
Thus $\sum_{Q_j \in \mathscr{Q}_{n_j}: j = 0, \dots, d-1} \Pi_{j = 0}^{d-1} [B_j (Q_j P_j) A_j (Q_{j+1 \Mod{d}}P_{j+1 \Mod{d}})^{-1}C_j]_{X_j,Y_j} =$\\
$\delta_{\left\{k_0 = k_1 = \dots = k_{d-1} \right\}} (\mathds{1}+\mathds{1})^{\sum_{j = 0}^{d-1} n_j}\sum_{S_{1,j} \in \binom{[n_{j}]}{k_j}: j = 0, \dots, d-1}$\\
$\Pi_{j = 0}^{d-1} [B_j]_{X_j,S_{1,j}}[P_j A_j P_{j+1 \Mod{d}}^{-1}]_{S_{1,j},S_{1,j+1\Mod{d}}}[C_j]_{S_{1,j+1\Mod{d}},Y_j}$.

We reduce $\Pi_{j = 0}^{d-1} [P_j A_j P_{j+1 \Mod{d}}^{-1}]_{S_{1,j},S_{1,j+1\Mod{d}}} =$\\
$\sum_{R_{j} \in \binom{[n_j]}{k_j}, L_{j} \in \binom{[n_{j+1\Mod{d}}]}{k_{j+1\Mod{d}}}} \Pi_{j = 0}^{d-1} [P_j]_{S_{1,j},R_{j}} [A_j]_{R_{j},L_{j}} [P_{j+1 \Mod{d}}^{-1}]_{L_{j},S_{1,j+1\Mod{d}}}$.

Observe that $\sum_{P_j \in \mathscr{P}_{n_j}: j = 0, \dots, d-1} \Pi_{j = 0}^{d-1} [P_j A_j P_{j+1 \Mod{d}}^{-1}]_{S_{1,j},S_{1,j+1\Mod{d}}} =\\$
$(\Pi_{j = 0}^{d-1} k_j! (n_j-k_j)!) \sum_{R_{j} \in \binom{[n_j]}{k_j}, L_{j} \in \binom{[n_{j+1\Mod{d}}]}{k_{j+1\Mod{d}}}} \Pi_{j = 0}^{d-1} [A_j]_{R_{j},L_{j}} \delta_{\left\{R_{j} = L_{j-1\Mod{d}}: j = 0, \ldots, d-1 \right\}}=$\\
$(\Pi_{j = 0}^{d-1} k_j! (n_j-k_j)!) [\Pi_{j = 0}^{d-1} A_j]^{(k_0)}$.

Now combining everything, we get $\sum_{Q_j \in \mathscr{Q}_{n_j}: j = 0, \dots, d-1} \sum_{P_j \in \mathscr{P}_{n_j}: j = 0, \dots, d-1}$ \\
$\Pi_{j = 0}^{d-1} [B_j (Q_j P_j) A_j (Q_{j+1 \Mod{d}}P_{j+1 \Mod{d}})^{-1}C_j]_{X_j,Y_j} =$\\
$\delta_{\left\{k_0 = k_1 = \dots = k_{d-1} \right\}} (\mathds{1}+\mathds{1})^{\sum_{j = 0}^{d-1} n_j} [\Pi_{j = 0}^{d-1} A_j]^{(k_0)} (\Pi_{j = 0}^{d-1} k_j! (n_j-k_j)!) \sum_{S_{1,j} \in \binom{[n_{j}]}{k_j}: j = 0, \dots, d-1}$\\
$\Pi_{j = 0}^{d-1} [B_j]_{X_j,S_{1,j}} [C_j]_{S_{1,j+1\Mod{d}},Y_j}=$\\
$\delta_{\left\{k_0 = k_1 = \dots = k_{d-1} \right\}} (\mathds{1}+\mathds{1})^{\sum_{j = 0}^{d-1} n_j} (\Pi_{j = 0}^{d-1} k_j! (n_j-k_j)! [B_j C_{j-1 \Mod{d}}]_{X_j,Y_{j-1 \Mod{d}}}) [\Pi_{j = 0}^{d-1} A_j]^{(k_0)}$ for all $X_j \in \binom{[p_j]}{k_j}$, $Y_j \in \binom{[r_j]}{k_j}$.
\end{proof}

For the next corollary, we use this lemma from \cite{marc}.

\begin{lem}\cite{marc}\label{coeff}
Let $A \in \mathbb{K}^{n \times n}$. Then $det[xI+A] = \sum_{i = 0}^{n} x^{n-i} [A]^{(i)}$ as a polynomial in $\mathbb{K}[x]$. 
\end{lem}

This next corollary generalizes Section $5$ of \cite{marc}, which had formulas for the cases when $d \leq 2$. 


\begin{cor}
For $j = 0, \ldots, d-1$, let $A_{j},D_{j} \in \mathbb{K}^{n_{j} \times n_{j+1 \Mod{d}}}$. For $Q_j \in \mathscr{Q}_{n_j}$ and $P_j \in \mathscr{P}_{n_j}$, let $P = (P_0, \dots, P_{d-1})$ and let $Q = (Q_0, \dots, Q_{d-1})$. Define the matrices $M(P,Q) = \Pi_{j = 0}^{d-1} ((Q_j P_j) A_j (Q_{j+1 \Mod{d}}P_{j+1 \Mod{d}})^{-1}+D_j)$. Let $det[xI+\Pi_{j=0}^{d-1} A_{j}] = \sum_{i = 0}^{n_0} x^{n_0-i} p_i$, $det[xI+\Pi_{j=0}^{d-1} D_{j}] = \sum_{i = 0}^{n_0} x^{n_0-i} q_i$, and $\sum_{P,Q} det[xI+M(P,Q)] = \sum_{i = 0}^{n_0} x^{n_0-i} r_i$ be polynomials in $\mathbb{K}(x)$. Then $r_k = (\mathds{1}+\mathds{1})^{\sum_{j = 0}^{d-1} n_j} \sum_{i = 0}^{k} (\Pi_{j = 0}^{d-1} \frac{(n_j-k+i)!(n_j-i)!}{(n_j-k)!}) p_i q_{k-i}$ for all $0 \leq k \leq n_0$.
\end{cor}

\begin{proof}
By Lemma \ref{coeff}, it is equivalent to show that $\sum_{P,Q} [M(P,Q)]^{(k)} =  (\mathds{1}+\mathds{1})^{\sum_{j = 0}^{d-1} n_j} \sum_{i = 0}^{k} (\Pi_{j = 0}^{d-1} \frac{(n_j-k+i)!(n_j-i)!}{(n_j-k)!}) [\Pi_{j=0}^{d-1} A_{j}]^{(i)} [\Pi_{j=0}^{d-1} D_{j}]^{(k-i)}$.
First we observe that $[M(P,Q)]^{(k)} = $\\
$\sum_{X_j \in \binom{[n_j]}{k}: j = 0, \dots, d-1} \Pi_{j = 0}^{d-1} [(Q_j P_j) A_j (Q_{j+1 \Mod{d}}P_{j+1 \Mod{d}})^{-1}+D_j]_{X_j, X_{j+1 \Mod{d}}}$, where $[(Q_j P_j) A_j (Q_{j+1 \Mod{d}}P_{j+1 \Mod{d}})^{-1}+D_j]_{X_j, X_{j+1 \Mod{d}}} =$\\
$\sum_{i_j = 0}^{k} \sum_{U_j,V_j \in \binom{[k]}{i_j}}$\\
$(-\mathds{1})^{||U_j+V_j||_1}[Q_j P_j A_j P_{j+1 \Mod{d}}^{-1} Q_{j+1 \Mod{d}}]_{U_j(X_j),V_j(X_{j+1 \Mod{d}})}[D_j]_{\overline{U_j}(X_j),\overline{V_j}(X_{j+1 \Mod{d}})}$.

By Lemma \ref{gen}, $\sum_{P,Q} \Pi_{j = 0}^{d-1} [Q_j P_j A_j P_{j+1 \Mod{d}}^{-1}Q_{j+1 \Mod{d}}]_{U_j(X_j),V_j(X_{j+1 \Mod{d}})} =$\\
$\delta_{\left\{i_0 = i_1 = \dots = i_{d-1}\right\}} (\mathds{1}+\mathds{1})^{\sum_{j = 0}^{d-1} n_j} (\Pi_{j = 0}^{d-1} i_j! (n_j-i_j)! [I_{n_j}]_{U_j(X_j),V_{j-1 \Mod{d}}(X_j)}) [\Pi_{j = 0}^{d-1} A_j]^{(i_0)} =$\\
$\delta_{\left\{i_0 = i_1 = \dots = i_{d-1}\right\}} (\mathds{1}+\mathds{1})^{\sum_{j = 0}^{d-1} n_j} (\Pi_{j = 0}^{d-1} i_j! (n_j-i_j)! \delta_{\left\{U_j = V_{j-1 \Mod{d}} \right\}}) [\Pi_{j = 0}^{d-1} A_j]^{(i_0)}$.

Thus $\sum_{P,Q} [M(P,Q)]^{(k)} = $\\
$(\mathds{1}+\mathds{1})^{\sum_{j = 0}^{d-1} n_j} \sum_{i = 0}^{k} (\Pi_{j = 0}^{d-1} i!(n_j-i)!)[\Pi_{j=0}^{d-1} A_{j}]^{(i)}\sum_{X_j \in \binom{[n_j]}{k}: j = 0, \dots, d-1}\sum_{U_j \in \binom{[k]}{i}: j = 0, \dots, d-1}$\\
$\Pi_{j = 0}^{d-1} [D_j]_{\overline{U_j}(X_j),\overline{U_{j+1 \Mod{d}}}(X_{j+1 \Mod{d}})}$.

As in \cite{marc}, we write \\
$\sum_{X_j \in \binom{[n_j]}{k}: j = 0, \dots, d-1}\sum_{U_j \in \binom{[k]}{i}: j = 0, \dots, d-1}\Pi_{j = 0}^{d-1} [D_j]_{\overline{U_j}(X_j),\overline{U_{j+1 \Mod{d}}}(X_{j+1 \Mod{d}})}=$
$\sum_{X_j \in \binom{[n_j]}{k}: j = 0, \dots, d-1}\sum_{W_j \in \binom{[n_j]}{k-i}: j = 0, \dots, d-1} \Pi_{j = 0}^{d-1} [D_j]_{W_j,W_{j+1 \Mod{d}}}\delta_{W_j \subset X_j} =$\\
$\sum_{W_j \in \binom{[n_j]}{k-i}: j = 0, \dots, d-1} (\Pi_{j = 0}^{d-1} \binom{n_j-k+i}{i} [D_j]_{W_j,W_{j+1 \Mod{d}}}) =$\\
$(\Pi_{j = 0}^{d-1} \binom{n_j-k+i}{i}) [\Pi_{j = 0}^{d-1} D_j]^{(k-i)}$, finishing the proof.
\end{proof}

We finish with an unsolved question from Marcus: 

\begin{que}
Is it possible to derive versions of the last few results in this paper with permutation matrices $P$ in the sums but no sign matrices $Q$?
\end{que}

This question is open even in the case that $d = 2$.

\section{Acknowledgments}
Thanks to Adam Marcus for helpful ideas and questions about discrete unitary invariance.

\end{document}